\documentclass[11pt]{article}
\usepackage{amsmath}
\usepackage{amstext}
\usepackage{amsthm}
\usepackage{amssymb}
\usepackage{graphicx,epsf}

\renewcommand{\Re}{\mathbb{R}}

\newtheorem{theorem}{Theorem}
\newtheorem{proposition}{Proposition}
\newtheorem{corollary}{Corollary}
\newtheorem{lemma}{Lemma}

\def\Re{\mathbb{R}}

\newcommand{\varep}{\varepsilon}

\def\cU{{\cal U}}
\def\cV{{\cal V}}
\def\cW{{\cal W}}

\def\cZ{{\cal Z}}

\newcommand{\tcU}{{\tilde \cU}}
\newcommand{\tcV}{{\tilde \cV}}

\newcommand{\tf}{{\tilde{f}}}

\newcommand{\tu}{{\tilde{u}}}

\newcommand{\tw}{{\tilde{w}}}
\newcommand{\tx}{{\tilde{x}}}
\newcommand{\ty}{{\tilde{y}}}
\newcommand{\tz}{{\tilde{z}}}

\newcommand{\tD}{{\tilde{D}}}

\newcommand{\tH}{{\tilde{H}}}
\newcommand{\tK}{{\tilde{K}}}

\newcommand{\tU}{{\tilde{U}}}
\newcommand{\tV}{{\tilde{V}}}

\newcommand{\tX}{{\tilde{X}}}
\newcommand{\tY}{{\tilde{Y}}}

\def\varep{\varepsilon}

\title{A Geometric Vietoris-Begle Theorem, with an Application to
  Convex Subsets of Topological Vector Lattices}

\author{Andrew McLennan\thanks{School of Economics, University of
    Queensland, Level 6 Colin Clark Building, St Lucia QLD 4072,
    Australia, \texttt{a.mclennan@economics.uq.edu.au}.  I am
    grateful to Rabee Tourky for guidance concerning Riesz spaces, and
    it is a pleasure to acknowledge a helpful conversation with
    Benjamin Burton. \emph{2010 Mathematics Subject Classification:}
    Primary 55P10; Secondary 46A40, 54C55, 55M15.}}

\begin{document}
\maketitle

\begin{abstract} 
  We show that if $L$ is a topological vector lattice, $u \colon L \to
  L$ is the function $u(x) = x \vee 0$, $C \subset L$ is convex, and
  $D = u(C)$ is metrizable, then $D$ is an ANR and $u|_C \colon C \to
  D$ is a homotopy equivalence, so $D$ is contractible and thus an AR.
  This is proved by verifying the hypotheses of a second result: if
  $X$ is a connected space that is homotopy equivalent to an ANR, $Y$
  is an ANR, and $f \colon X \to Y$ is a continuous surjection such
  that, for each $y \in Y$ and each neighborhood $V \subset Y$ of $y$,
  there is a neighborhood $V' \subset V$ of $y$ such that $f^{-1}(V')$
  can be contracted in $f^{-1}(V)$, then $f$ is a homotopy
  equivalence. The latter result is a geometric analogue of the
  Vietoris-Begle theorem.
\end{abstract}


\section{Introduction} \label{sec:Introduction}

The project this paper reports on began with a seemingly simple
question.  Let $u \colon \Re^n \to \Re^n_+$ be the function $u(x) =
(\max\{x_1,0\}, \ldots, \max\{x_n,0\})$.  If $C \subset \Re^n$ is
convex and $D = u(C)$, is $D$ contractible?  Easily visualized
examples such as a disk pierced by the corner of $\Re^3_+$ suggest an
affirmative answer.  It is not hard to prove that $D$ is
contractible when $C$ is a line or line segment, but the argument does
not easily generalize.

To see why this question might appeal to a mathematical economist, we
review some fixed point theory.  If $X$ and $Y$ are topological
spaces, a \emph{correspondence} $F \colon X \to Y$ is an assignment of
a nonempty $F(x) \subset Y$ to each $x \in X$. Such an $F$ is
\emph{convex} (compact, etc.) \emph{valued} if each $F(x)$ is convex
(compact, etc.), and $F$ is \emph{upper hemicontinuous} if, for each
$x \in X$ and each open $V \subset Y$ containing $F(x)$, there is a
neighborhood $U$ of $x$ such that $F(x') \subset V$ for all $x' \in
U$.  The Kakutani fixed point theorem asserts that if $C \subset
\Re^n$ is nonempty, compact, and convex, and $F \colon C \to C$ is an upper
hemicontinuous, compact, convex valued correspondence, then there is
an $x^* \in C$ such that $x^* \in F(x^*)$.  This extension of the
Brouwer fixed point theorem is used to prove the fundamental
equilibrium existence results for game theory and general economic
equilibrium, and it (and its infinite dimensional extensions) are
frequently applied throughout economic theory.

Insofar as the conclusion of the Kakutani fixed point theorem is
topological, the geometric hypotheses seem perhaps too strong.
Eilenberg and Montgomery \cite{EiMo46} showed how they can be relaxed.
A space $Z$ is \emph{acyclic} with respect to a homology theory $H_*$
with associated reduced homology $\tH_*$ (cohomology theory $H^*$ with
associated reduced cohomology $\tH^*$) if $\tH_n(Z) = 0$ ($\tH^n(Z) =
0$) for all $n = 0, 1, 2, \ldots$.  Recall that a metric space $X$ is
an \emph{absolute retract} (AR) if, whenever $e \colon X \to Z$ is an
embedding of $X$ as a closed subset of a metric space $Z$, there is a
retraction $r \colon Z \to e(X)$.  The Eilenberg-Montgomery fixed
point theorem asserts that if $C$ is a nonempty compact AR and $F \colon C
\to C$ is an upper hemicontinuous correspondence that is compact and
acyclic (for Vietoris homology) valued, then there is an $x^* \in C$
such that $x^* \in F(x^*)$.

Since homology and cohomology are invariant under homotopy, a
contractible space is acyclic for any homology or cohomology theory.
We can imagine composing a convex valued correspondence with $u$ to
obtain a contractible valued correspondence, then applying the
Eilenberg-Montgomery theorem.  To be honest, what I initially thought
might be an application of this sort turned out to be a mirage, and I
know of no actual economic application, but nevertheless the issue
still seems quite interesting.

We can place our problem in a more general setting.  A \emph{vector
  lattice} or \emph{Riesz space} is a vector space $L$ over the reals
endowed with a partial order $\ge$ such that: a) if $x \ge y$, then $x
+ z \ge y + z$ for all $z \in L$ and $\alpha x \ge \alpha y$ for all
$\alpha \ge 0$; b) any two elements $x,y \in L$ have a least upper
bound $x \vee y$ and a greatest lower bound $x \wedge y$.  Fix such an
$L$.  The \emph{lattice cone} of $L$ is $L_+ = \{\, x \in L : x \ge 0
\,\}$.  Let $u \colon L \to L_+$ be the function $u(x) = x \vee 0$.
(For any $a \in L$ all of our results hold equally, with obvious
modifications, for the functions $x \mapsto x \vee a$ and $x \mapsto x
\wedge a$.)

An important basic result (\cite[p.~5]{AlBu03}) is that $L$ is a
\emph{distributive lattice}:
$$x \vee (y \wedge z) = (x \vee y) \wedge (x \vee z) \quad \text{and}
\quad x \wedge (y \vee z) = (x \wedge y) \vee (x \wedge z)$$ for all
$x, y, z \in L$.  Two other basic results will be important.  Let
$|\cdot| \colon L \to L_+$ be the function $|x| = (x \vee 0) - (x
\wedge 0)$.

\begin{lemma} \label{lemma:UpperLower}
  For all $x_0, x_1 \in L$ and $t \in [0,1]$, $$u(x_0) - |u(x_1) -
  u(x_0)| \le u(x_0) \wedge u(x_1) \le u((1 - t)x_0 + tx_1) \le u(x_0)
  \vee u(x_1) \le u(x_0) + |u(x_1) - u(x_0)|.$$
\end{lemma}

\begin{proof}
  For the first inequality we have the general computation
  $$y_0 - |y_1 - y_0| = y_0 + (y_1 - y_0) \wedge 0 - (y_1 - y_0) \vee
  0 \le y_0 + (y_1 - y_0) \wedge 0 = y_1 \wedge y_0.$$
 For the second inequality we compute that
  $$u(x_0) \wedge u(x_1) = (x_0 \vee 0) \wedge (x_1 \vee 0) = (x_0
  \wedge x_1) \vee 0 = ((1 - t)(x_0
  \wedge x_1) + t(x_0
  \wedge x_1)) \vee 0$$
  $$\le ((1 - t)x_0 + tx_1) \vee 0 = u((1 - t)x_0 + tx_1).$$
  The third and fourth inequalities follow from symmetric computations.
\end{proof}

\begin{corollary}
  For each $y \in L_+$, $u^{-1}(y)$ is convex.
\end{corollary}

A set $A \subset L$ is \emph{solid} if, for all $y \in A$, $A$
contains all $x \in L$ such that $|x| \le |y|$.  If, in addition to
being a Riesz space, $L$ is a (not necessarily Hausdorff) topological
vector space and its topology has a base at the origin consisting of
solid sets, then $L$ is \emph{locally solid}, and a \emph{topological
  vector lattice}. A result of Roberts and Namioka
(\cite[p.~55]{AlBu03}) asserts that $L$ is locally solid if and only
if the function $(x,y) \mapsto x \vee y$ is uniformly
continuous\footnote{That is, for any neighborhood of the origin $V$
  there is a neighborhood of the origin $U$ such that $x' \vee y' \in
  (x \vee y) + V$ for all $x,y,x',y'$ such that $x' \in x + U$ and $y'
  \in y + U$.}, and this is the case if and only if the function
$(x,y) \mapsto x \wedge y$ is uniformly continuous. (It is possible
(\cite[p.~56]{AlBu03}) that $u$ is continuous even when $L$ is not
locally solid.)  From this point forward we assume that $L$ is a
topological vector lattice.

Recall that a metric space $X$ is an \emph{absolute neighborhood
  retract} (ANR) if, whenever $e \colon X \to Z$ is an embedding of
$X$ as a closed subset of another metric space $Z$, there is a
neighborhood $U \subset Z$ of $e(X)$ and a retraction $r \colon U \to
e(X)$.  An ANR is an AR if and only if it is contractible
(\cite[11.2]{GrDu03}).  If $C \subset L$ is convex, then for any $x_0
\in C$ the contraction $c \colon C \times [0,1] \to C$ given by
$c(x,t) = (1 - t)x + tx_0$ is continuous by virtue of the continuity
of the vector operations (even if $L$ is not locally convex) so $C$ is
contractible, and thus an AR if it is an ANR.

Our first main result is:

\begin{theorem} \label{th:MainOne}
  If $C \subset L$ is convex and $D = u(C)$ is metrizable, then $D$ is
  an ANR and $u|_C \colon C \to D$ is a homotopy equivalence, so $D$
  is an AR.
\end{theorem}

\smallskip
\noindent \textbf{Remark:} Among various ways that $D$ may be
metrizable even if $L$ is not, we mention that Varadarajan
\cite{Var58} has shown that if $L$ is the space of measures on a
compact metric space with the weak topology, then $L_+$ is metrizable,
but $L$ itself is metrizable only under quite restrictive conditions.
This case is common in economic applications.

\smallskip
The next result is one of the key ideas of the proof of Theorem
\ref{th:MainOne}.

\begin{proposition} \label{prop:SolidPreimage}
  If $y \in L_+$ and $U \subset L$ is a neighborhood of the origin, then
  there is a neighborhood $W \subset U$ of the origin such that the convex
  hull of $u^{-1}(y + W)$ is contained in $u^{-1}(y + U)$.
\end{proposition}

\begin{lemma}  \label{lemma:Interval}
  For any neighborhood $U \subset L$ of the origin there is a
  neighborhood $V \subset U$ of the origin such that $w \in U$
  whenever $u, v \in V$, $w \in L$, and $u \le w \le v$.
\end{lemma}

\begin{proof}
  Without loss of generality we may assume that $U$ is solid.  Since
  $\vee$, $\wedge$, and the vector operations are continuous, there is
  a neighborhood $V$ of the origin such that $-(u \wedge 0) + (v \vee
  0) \in U$ for all $u, v \in V$.  If $u, v \in V$, $w \in L$, and $u
  \le w \le v$, then $u \wedge 0 \le w \wedge 0$ and $w \vee 0 \le v
  \vee 0$, so $|w| \le -(u \wedge 0) + (v \vee 0)$ and thus $w \in U$.
\end{proof}

\begin{proof}[Proof of Proposition \ref{prop:SolidPreimage}]
  Lemma \ref{lemma:Interval} gives a neighborhood $V \subset U$ of the
  origin such that $w \in U$ whenever $u, v \in V$ and $u \le w \le
  v$.  Since $\vee$, $\wedge$, and the vector operations are
  continuous, there is a neighborhood $W$ of the origin such that $y_0
  - |y_0 - y_1| - y, y_0 + |y_0 - y_1| - y \in V$ for all $y_0, y_1
  \in y + W$.  For any $x_0, x_1 \in u^{-1}(y + W)$ and any $t \in
      [0,1]$ we have $u(x_0) - |u(x_0) - u(x_1)|, u(x_0) + |u(x_0) -
      u(x_1)| \in y + V$ and therefore $u((1 - t)x_0 + tx_1) \in y +
      U$ by virtue of Lemma \ref{lemma:UpperLower}.
\end{proof}

\section{A Sufficient Condition for Homotopy Equivalence}

Following Milnor \cite{Mil59} let $\cW_0$ be the class of spaces which
have the homotopy type of a countable CW-complex.  As Milnor explains,
results of Whitehead \cite{Whi49} and Hanner \cite{Han51} imply that
$\cW_0$ is also the class of all spaces that have the homotopy type of
a separable ANR, the class of all spaces that have the homotopy type
of a countable locally finite simplicial complex, and the class of all
spaces that are dominated\footnote{A space $P$ \emph{dominates} a
  space $X$ if there are maps $f \colon X \to P$ and $g \colon P \to
  X$ such that $gf \simeq 1$.} by a countable CW-complex. 

Let $X$ and $Y$ be topological spaces, and let $f \colon X \to Y$ be a
map, A \emph{compressive pair} for $f$ is a pair $(V',V)$ where $V$
and $V'$ are open subsets of $Y$ with $V' \subset V$ such that
$f^{-1}(V')$ is contractible in $f^{-1}(V)$: there is a continuous
$\xi \colon f^{-1}(V') \times [0,1] \to f^{-1}(V)$ such that
$\xi(\cdot,0)$ is the identity function of $f^{-1}(V')$ and
$\xi(\cdot,1)$ is a constant function.  We say that $f$ is
\emph{compressive} if it is surjective and, for every $y \in Y$ and
every neighborhood $V \subset Y$ of $y$, there is a neighborhood $V'$
of $y$ such that $(V',V)$ is a compressive pair.  Evidently
Proposition \ref{prop:SolidPreimage} implies that:

\begin{lemma} \label{lemma:Compressive}
  $u|_C \colon C \to D$ is compressive.
\end{lemma}

Theorem \ref{th:MainOne} will follow from this and Theorem
\ref{th:MainTwo}, which is our second main result.  Prior results
similar to Theorem \ref{th:MainTwo} include various results in Section
3 of \cite{Arm69} and Proposition 2.1.8 of \cite{WJR13}.  In all of
those results $X$ and $Y$ are assumed to be compact and finite
dimensional, and the assumption imposed on the fibers of $f$ is (in
effect) that they are compact AR's.

\begin{theorem} \label{th:MainTwo}
  If $X$ is a connected element of $\cW_0$, $Y$ is an ANR, and $f
  \colon X \to Y$ is a compressive surjection, then $f$ is a homotopy
  equivalence.
\end{theorem}

Fix $X$, $Y$, and $f \colon X \to Y$ satisfying the hypotheses of
Theorem \ref{th:MainTwo}.  The remainder of this section proves
Theorem \ref{th:MainTwo} by verifying the hypotheses of the following
result, which is Theorem 1 of \cite{Whi49} with slightly less refined
hypotheses.  It has come to be known as \emph{Whitehead's
  theorem}\footnote{Hatcher \cite[Ch.~4]{Hat02} presents an accessible
  proof of Proposition \ref{prop:WhiteheadOne} when $X$ and $Y$ are
  CW-complexes, and \cite[Prop. A.11]{Hat02} also shows that a space
  that is dominated by a CW-complex is homotopy equivalent to a
  CW-complex.}.

\begin{proposition} \label{prop:WhiteheadOne}
  If $X$ and $Y$ are connected elements of $\cW_0$, then a surjective
  map $f \colon X \to Y$ is a homotopy equivalence if and only if $f_*
  \colon \pi_n(X) \to \pi_n(Y)$ is an isomorphism for all $n = 1, 2,
  \ldots$.
\end{proposition}

A \emph{compressive cover} for $f$ is a collection $\cV$ of
compressive pairs such that $\{\, V' : (V',V) \in \cV \,\}$ is a
locally finite cover of $Y$.  For any open cover $\cU$ of $Y$ there is
a compressive cover $\cV$ such that for all $(V',V) \in \cV$, $V \in
\cU$: since $f$ is compressive, there is a collection $\cZ$ of
compressive pairs $(V',V)$ such that $V \in \cU$ and $\{\, V' : (V',V)
\in \cZ \,\}$ is a cover of $Y$, and since $Y$ is metric, hence
paracompact, $\{\, V' : (V',V) \in \cZ \,\}$ has a locally finite
refinement.  We say that a compressive cover $\tcV$ is a \emph{star
  refinement} of $\cV$ if, for each $(\tV',\tV) \in \tcV$, there is a
$(V',V) \in \cV$ such that
$$\bigcup_{(\tV'_0,\tV_0) \in \tcV, \tV_0 \cap \tV \ne \emptyset}
\tV_0 \subset V'.$$

\begin{lemma}
  For any compressive cover $\cV$ there is an compressive cover $\tcV$
  that star refines it.
\end{lemma}

\begin{proof}
  Since $Y$ is metric, \cite[VIII.3]{Du66} gives a refinement $\cU$ of
  $\{\, V' : (V',V) \in \cV \,\}$ such that for each $U \in \cU$,
  $\bigcup_{U' \in \cU, U \cap U' \ne \emptyset} U' \subset V'$ for
  some $(V',V) \in \cV$.  As we pointed out above, there is a
  compressive cover $\tcV$ such that for all $(\tV',\tV) \in \tcV$,
  $\tV \in \cU$.
\end{proof}

The Arens-Eells embedding theorem (e.g., \cite[p.~597]{GrDu03})
implies that $Y$ can be isometrically embedded as a closed subset of a
normed linear space $N_Y$.  Fix a retraction $s_Y \colon V_Y \to Y$ of
a neighborhood $V_Y \subset N_Y$ of $Y$.  Let $W_Y \subset Y \times Y$
be a neighborhood of the diagonal $\{\, (y,y) : y \in Y \,\}$ such
that $(1 - t)y_0 + ty_1 \in V_Y$ for all $(y_0,y_1) \in W_Y$ and $t
\in [0,1]$.

For each positive integer $k$ let $D^k$ and $S^{k-1}$ be the closed
unit ball in $\Re^k$ and its boundary.

\begin{proposition} \label{prop:Lift}
  Let $K$ be a finite simplicial complex, let $J$ be a subcomplex, and
  let $g \colon K \to Y$ and $\eta \colon J \to X$ be maps such that $f\eta =
  g|_J$.  Then there is a continuous extension $\gamma \colon K \to X$ of
  $\eta$ such that $(g(p),f(\gamma(p))) \in W_Y$ for all $p \in K$.
\end{proposition}

\begin{proof}
  (The following construction also occurs in \cite[p.~436]{ArPr69}.)
  Let $n$ be the dimension of $K$.  Let $\cV_n$ be a compressive cover
  such that for all $(V',V) \in \cV_n$, $V \times V \subset W_Y$.
  Choose compressive covers $\cV_{n-1}, \ldots, \cV_0$ such that each
  $\cV_k$ is a star refinement of $\cV_{k+1}$.  In the usual way we
  let the vertices of $K$ be the standard unit basis vectors of a
  Euclidean space, and we let each simplex be the convex hull of its
  vertices.  The Lebesgue number lemma implies that for each $k \ge 1$
  there is $\varep_k > 0$ such that for any point in $K$ the ball of
  radius $\varep_k$ around that point is contained in $g^{-1}(V)$ for
  some $(V',V) \in \cV_k$.  After sufficient repeated barycentric
  subdivision (e.g., \cite{Do80}) the mesh of $K$ is less than
  $\min\{\varep_1, \ldots, \varep_n\}$, so we may assume that for each
  $k$-simplex $\sigma$ of $K$ there is some $(V',V) \in \cV_k$ such
  that $g(\sigma) \subset V$.

  Let $K^{(k)}$ be the $k$-skeleton of $K$.  Proceeding inductively,
  for $k = 0, \ldots, n$ we will construct extensions $\gamma_k \colon
  J \cup K^{(k)} \to X$ of $\eta$ such that for each $k$-simplex
  $\sigma$ of $K$ there is some $(V',V) \in \cV_k$ such that
  $\gamma(\sigma) \subset f^{-1}(V)$.  First construct an extension
  $\gamma_0 \colon J \cup K^{(0)} \to X$ by letting the image
  $\gamma_0(v)$ of a vertex $v$ of $K$ that is not in $J$ be any
  element of $f^{-1}(g(v))$.  (Of course $\gamma_0(v) \in f^{-1}(V)$
  for any $(V',V) \in \cV_0$ such that $g(v) \in V$.)

  Now suppose that $\gamma_{k-1}$ has already been constructed,
  $\sigma$ is a $k$-simplex of $K$ that is not in $J$, and $\tau$ is a
  facet of $\sigma$.  Let $(\tV',\tV)$ be an element of $\cV_{k-1}$
  such that $g(\tau) \subset \tV$.  Since every facet of $\sigma$
  intersects $\tau$ and $\cV_{k-1}$ is a star refinement of $\cV_k$,
  there is a $(V',V) \in \cV_k$ such that $g(\partial \sigma) \subset
  V'$.  Let $\xi \colon f^{-1}(V') \times [0,1] \to f^{-1}(V)$ be
  continuous with $\xi(\cdot,0)$ the identity function of $f^{-1}(V')$
  and $\xi(\cdot,1)$ a constant function. For any homeomorphism $h \colon
  D^k \to \sigma$ such that $h(S^{k-1}) = \partial \sigma$,
  $\gamma_{k-1}|_{\partial \sigma}$ has an extension $\gamma_k|_\sigma
  \colon \sigma \to f^{-1}(V)$ (which is obviously continuous) given
  by setting $$\gamma_k|_\sigma(h((1 - t)p)) =
  \xi(\gamma_{k-1}(h(p),t))$$ for all $p \in S^{k-1}$ and $t \in
     [0,1)$, and letting $\gamma_k|_\sigma(h(0))$ be the constant
       value of $\xi(\cdot,1)$.  Combining such extensions for all
       $k$-simplices in $K$ that are not in $J$ constructs $\gamma_k$.
       Finally let $\gamma = \gamma_n$.
\end{proof}

\begin{proof}[Proof of Theorem \ref{th:MainTwo}]
  For any $n \ge 1$ and continuous $g \colon S^n \to Y$ the last result
  gives a continuous $\gamma \colon S^n \to X$ such that
  $(g(p),f(\gamma(p))) \in W_Y$ for all $p \in S^n$, so that $h(p,t) =
  s_Y((1 - t)g(t) + tf(\gamma(p)))$ is a homotopy between $g$ and
  $f\gamma$.  Thus $f_* \colon \pi_n(X) \to \pi_n(Y)$ is surjective.  If
  $\eta \colon S^n \to X$ is continuous and $g \colon D^{n+1} \to Y$ is an
  extension of $f\eta$, then the last result gives a continuous
  extension $\gamma \colon D^{n+1} \to X$ of $\eta$.  Thus $f_* \colon \pi_n(X)
  \to \pi_n(Y)$ is also injective, so Proposition
  \ref{prop:WhiteheadOne} implies that $f$ is a homotopy equivalence.
\end{proof}

\section{The Proof of Theorem \ref{th:MainOne}} \label{sec:ProofThree}

A convex subset of a topological vector space is contractible (even if
the TVS is not locally convex, because the vector operations are
continuous) so $C$ is an element of $\cW_0$.  In view of Lemma
\ref{lemma:Compressive}, once we have shown that $D$ is an ANR,
Theorem \ref{th:MainTwo} implies the other assertions of Theorem
\ref{th:MainOne}.  The following sufficient condition for a space to
be an ANR is well known (e.g., \cite[Prop.~8.3]{McL18}).

\begin{lemma} \label{lemma:Sufficient}
  If $Z$ is a Hausdorff locally convex topological vector space, $K
  \subset Z$ is convex, $U \subset K$ is (relatively) open, $A \subset
  U$ is metrizable, and $r \colon U \to A$ is a retraction, then $A$
  is an ANR\footnote{For the sake of self containment we include the
    proof.  Suppose that $X$ is a metric space, $e \colon A \to X$
    maps $A$ homeomorphically onto $e(A)$, which is closed.  A
    generalization of the Tietze extension theorem due to Dugundji
    \cite{Du51} implies that $e^{-1} \colon e(A) \to A$ has a
    continuous extension $j \colon X \to K$. (Dugundji's proof is a
    variant of our proof that $D$ is an ANR, and after reading that
    the reader may have little difficulty constructing his argument.)
    Then $V = j^{-1}(U)$ is a neighborhood of $e(A)$, and $e \circ r
    \circ j|_V \colon V \to e(A)$ is a retraction.}.
\end{lemma}

Since $D$ is metrizable, the Arens-Eells theorem implies that there is
an embedding $e \colon D \to N$ of $D$ in a normed linear space $N$
such that $\tD = e(D)$ is closed in $N$.  Let $\tu = e \circ u|_C$,
and let $\tK$ be the convex hull of $\tD$.  Since a metric space is
paracompact, the open cover of $\tK \setminus \tD$ whose elements are
the balls (in $\tK$) centered at the various $\tx \in \tK \setminus D$
of radius one third of the distance from $\tx$ to $\tD$, has a locally
finite refinement $\tcU$. For each $\tU \in \tcU$ choose an $x_\tU \in
C$ such that the distance from $\tU$ to $\tu(x_\tU)$ is less than
twice the distance from $\tU$ to $\tD$.  Let $\{\varphi_\tU\}_{\tU \in
  \tcU}$ be a partition of unity subordinate to $\tcU$, define $\rho
\colon \tK \setminus \tD \to \tD$ by setting
$$\rho(z) = \tu\big(\sum_\tU \varphi_\tU(z)x_\tU\big),$$ and let $r
\colon \tK \to \tD$ be the function that is the identity on $D$ and
$\rho$ on $\tK \setminus \tD$. Lemma \ref{lemma:Sufficient} implies
that $D$ and $\tD$ are ANR's if $r$ is a retraction, which is
evidently the case if it is continuous.  Evidently $r$ is continuous
at each point in $\tK \setminus \tD$.  Fixing a point $\ty \in \tD$
and a neighborhood $\tV \subset \tD$, our goal is to find a
neighborhood $\tV'' \subset \tK$ of $\ty$ such that $r(\tV'') \subset
\tV$.

Proposition \ref{prop:SolidPreimage} gives a neighborhood $\tV'
\subset \tV$ such that $\tu^{-1}(\tV)$ contains the convex hull of
$\tu^{-1}(\tV')$.  Let $\delta > 0$ be small enough that $\tV'$
contains the ball of radius $\delta$ (in $\tD$) centered at $\ty$, and
let $\tV''$ be the ball of radius $\delta/6$ (in $\tK$) centered at
$\ty$.  Fix a point $\tz \in \tV''$.  If $\tz \in \tD$, then $r(\tz) =
\tz \in \tV'' \cap \tD \subset \tV$, so assume that $\tz \in \tV''
\setminus \tD$.

Any $\tU \in \tcU$ that contains $\tz$ is contained in the ball
centered at some $\tx \in \tK \setminus \tD$ of radius $\alpha$, where
$\alpha$ is one third of the distance from $\tx$ to $\tD$.  Choose
$\tw \in \tU$ whose distance from $\tu(x_\tU)$ is less than twice the
distance from $\tU$ to $\tD$, and thus less than $8\alpha$.  The
distance from $\tz$ to $\tw$ is less than $2\alpha$, so the distance
from $\ty$ to $\tu(x_\tU)$ is less than the distance from $\ty$ to
$\tz$ plus $10\alpha$, and thus less than $6$ times the distance from
$\ty$ to $\tz$ because the latter quantity is greater than $2\alpha$.
Therefore $\tu(x_\tU) \in \tV'$ for all $\tU \in \tcU$ that contain
$\tz$, so $\sum_\tU \varphi_\tU(\tz)x_\tU$ is in the convex hull of
$\tu^{-1}(\tV')$ and thus $\rho(\tz) \in \tV$, as desired.

\section{Relation with the Vietoris-Begle Theorem}

We state two versions of the Vietoris-Begle theorem, which use
Alexander-Spanier cohomology and homology respectively.  The first
might be regarded as the ``standard'' version.  It asserts that if $X$
and $Y$ are paracompact Hausdorff spaces, $f \colon X \to Y$ is a
closed continuous surjection, and, for some $n \ge 0$,
$\tH^k(f^{-1}(y)) = 0$ for all $y \in Y$ and $k < n$, then $\tH^k(f)
\colon \tH^k(X) \to \tH^k(Y)$ is an isomorphism for $k < n$ and an
injection for $k = n$. A particularly elegant proof is given in
\cite{Law73}.  The second version (which we use in the proof of
Theorem \ref{th:MainThree}) is a dual result that was established by
Volovikov and Ahn \cite{VoAn84} and reproved by Dydak \cite{Dyd86}.
It asserts that if $X$ and $Y$ are compact metrizable spaces, $f$ is a
continuous surjection, and, for some $n \ge 0$, $\tH_k(f^{-1}(y)) = 0$
for all $y \in Y$ and $k < n$, then $\tH_k(f) \colon \tH_k(X) \to
\tH_k(Y)$ is an isomorphism for $k < n$ and a surjection for $k = n$.
In each case $\tH^*(f)$ or $\tH_*(f)$ is an isomorphism if the fibers
are acyclic.

Insofar as each passes from an assumption that the fibers of $f$ are,
in some sense, trivial, to a conclusion that $f$ is or induces an
isomorphism, Theorem \ref{th:MainTwo} and the Vietoris-Begle theorem
(including other versions in the literature) are quite similar.  In
this connection we should also mention the Main Theorem of
\cite{Sma57a}, which asserts that if $X$ and $Y$ are path connected,
locally compact, separable metric spaces, $X$ is $LC^n$, and for each
$y \in Y$, $f^{-1}(y)$ is $LC^{n-1}$ and $(n-1)$-connected\footnote{A
  path connected space $X$ is: (a) \emph{$n$-connected} if, for every
  $1 \le k \le n$, every map $S^k \to X$ has a continuous extension
  $D^{k+1} \to X$; (b) $LC^n$ if, for every $x \in X$ and neighborhood
  $U \subset X$, there is a neighborhood $V \subset U$ of $x$ such
  that for each $1 \le k \le n$, every map $S^k \to V$ has a
  continuous extension $D^{k+1} \to U$.}, then $Y$ is $LC^n$ and $f_*
\colon \pi_k(X) \to \pi_k(Y)$ is an isomorphism for all $k < n$ and a
surjection when $k = n$.

Theorem \ref{th:MainThree} further develops the relation between these
results.  In comparison with Theorem \ref{th:MainTwo}, it imposes an
hypothesis on the fibers themselves, rather than the preimages of
neighborhoods of points of $Y$, but the domain must be compact and
simply connected.

\begin{theorem} \label{th:MainThree}
  If $X$ and $Y$ are compact connected ANR's, $X$ is simply connected,
  $f \colon X \to Y$ is a continuous surjection, and, for each $y \in
  Y$, $f^{-1}(y)$ is contractible, then $f$ is a homotopy equivalence.
\end{theorem}

We need a technical result that is a variant of \cite[Lemma 1]{Sma57}
and \cite[Lemma 1]{Sma57a}.

\begin{lemma} \label{lemma:Nearby}
  If $A$ is a compact space, $B$ is a Hausdorff space, $g \colon A \to
  B$ is a map, $b \in B$, and $U \subset A$ is an open neighborhood of
  $g^{-1}(b)$, then there is an open $V \subset B$ containing $b$ such
  that $g^{-1}(V) \subset U$.
\end{lemma}

\begin{proof}
  Otherwise for each open $V \subset B$ containing $b$ there would be
  an $a_V \in A \setminus U$ such that $g(a_V) \in V$.  Since $A
  \setminus U$ is compact, a subnet of $\{a_V\}$ would converge to one
  of its elements, say $a$.  Since $B$ is Hausdorff, continuity gives
  $g(a) = b$, but then $a \in g^{-1}(b) \cap (A \setminus U) =
  \emptyset$.
\end{proof}

As before we assume that $X$ and $Y$ are isometrically embedded in
normed linear spaces $N_X$ and $N_Y$, we fix retractions retractions
$r_X \colon U_X \to X$ and $s_Y \colon V_Y \to Y$ of neighborhoods
$U_X \subset N_X$ and $V_Y \subset N_Y$ of $X$ and $Y$, and we let $W
\subset Y \times Y$ be a neighborhood of the diagonal $\{\, (y,y) : y
\in Y \,\}$ such that $(1 - t)y_0 + ty_1 \in V_Y$ for all $(y_0,y_1)
\in W$ and $t \in [0,1]$.

\begin{lemma} \label{lemma:Paths}
  If $y \in Y$ and $f^{-1}(y)$ is path connected, then, for any open
  $V \subset Y$ containing $y$, there is an open $V' \subset V$
  containing $y$ such that any function $\eta_\partial \colon \{-1,1\} \to
  f^{-1}(V')$ has a continuous extension $\eta \colon [-1,1] \to
  f^{-1}(V)$.
\end{lemma}

\begin{proof}
  Since $f^{-1}(y)$ is compact there is a $\delta > 0$ such that the
  open $\delta$-ball $B_\delta$ in $N_X$ around $f^{-1}(y)$ is
  contained in $r_X^{-1}(f^{-1}(V))$.  Lemma \ref{lemma:Nearby} gives
  an open $V' \subset V$ containing $y$ such that $f^{-1}(V') \subset
  B_\delta$.  Suppose $\eta_\partial$ is a function from $\{-1,1\}$ to
  $f^{-1}(V')$.  Choose $x_{-1}, x_1 \in f^{-1}(y)$ such that the
  distance from $\eta_\partial(-1)$ to $x_{-1}$ and the distance from
  $\eta_\partial(1)$ to $x_1$ are both less than $\delta$.  A
  satisfactory extension $\eta$ can be constructed by combining
  reparameterizations of the three paths $t \mapsto r_X((1 -
  t)\eta_\partial(-1) + tx_{-1})$, a path $\pi \colon [0,1] \to
  f^{-1}(y)$ with $\pi(-1) = x_{-1}$ and $\pi(1) = x_1$, and the path
  $t \mapsto r_X((1 - t)x_1 + t\eta_\partial(1))$.
\end{proof}

Fix a point $x_0 \in X$, and let $y_0 = f(x_0)$.  Let $\tX$ and $\tY$
be the universal covering spaces of $X$ and $Y$, with respect to the
base points $x_0$ and $y_0$, and let $\tf \colon \tX \to \tY$ be the
lift of $f$ with respect to these base points.  The proof of Theorem
\ref{th:MainThree} verifies the hypotheses of the following variant of
Proposition \ref{prop:WhiteheadOne}, which is Theorem 3 of
\cite{Whi49}.

\begin{proposition} \label{prop:WhiteheadTwo}
  If $X$ and $Y$ are connected elements of $\cW_0$, a map $f \colon X
  \to Y$ is a homotopy equivalence if and only if $f_* \colon \pi_1(X)
  \to \pi_1(Y)$ is an isomorphism and $\tH_n(\tf) \colon \tH_n(\tX)
  \to \tH_n(\tY)$ is an isomorphism for each $n = 2, 3, \ldots$.
\end{proposition}

\begin{proof}[Proof of Theorem \ref{th:MainThree}]
Lemma \ref{lemma:Paths} implies that there is an open cover $\{\, V' :
(V',V) \in \cV \}$ of $Y$ where $\cV$ is a collection of pairs
$(V',V)$ such that $V$ and $V'$ are open subsets of $Y$, $V \times V
\subset W$, $V' \subset V$, and any function $\eta_\partial \colon
\{-1,1\} \to f^{-1}(V')$ has a continuous extension $\eta \colon
    [-1,1] \to f^{-1}(V)$.  Consider a continuous $g \colon S^1 \to
    Y$.  After sufficient subdivision $S^1$ is a simplicial complex,
    each of whose $1$-simplices $\sigma$ satisfies $g(\partial\sigma)
    \subset V'$ for some $(V',V) \in \cV$. For each vertex $v$ of this
    complex choose $\gamma_0(v) \in f^{-1}(g(v))$.  As in the proof of
    Proposition \ref{prop:Lift}, there is an extension $\gamma_1
    \colon S^1 \to X$ of $\gamma_0$ such that $(g(p),f(\gamma(p))) \in
    W$ for all $p \in S^1$, so that $h(p,t) = s_Y((1 - t)g(t) +
    tf(\gamma(p)))$ is a homotopy between $g$ and $f\gamma$.  Thus
    $f_* \colon \pi_1(X) \to \pi_1(Y)$ is surjective.

Since $X$ is simply connected, it follows that $Y$ is simply
connected, so $f_* \colon \pi_1(X) \to \pi_1(Y)$ is an isomorphism and
$\tX$, $\tY$, and $\tf$ are (up to irrelevant formalities) just $X$,
$Y$, and $f$.  As we mentioned previously, since each fibre
$f^{-1}(y)$ is contractible, it is acyclic, and $\tX = X$ is compact,
so the dual Vietoris-Begle theorem of Volovikov-Ahn and Dydak implies
that $\tH_n(\tf)$ is an isomorphism for all $n \ge 2$, after which
Proposition \ref{prop:WhiteheadTwo} implies that $f$ is a homotopy
equivalence.  (The dual Vietoris-Begle theorem is specific to
Alexander-Spanier homology, and Whitehead uses singular homology.
However, it is well known that Alexander-Spanier homology agrees with
{\v C}ech homology on compact Hausdorff spaces, and {\v C}ech and
singular homology agree on ANR's \cite{Dug55,Kod55,Mar58}.)
\end{proof}

The assumptions of Theorem \ref{th:MainThree} imply those of Theorem
\ref{th:MainTwo} if each fiber is an AR.

\begin{lemma} \label{lemma:Pairs}
  If $X$ is a locally compact metric space, $f \colon X \to Y$ is a
  surjective map, $V \subset Y$ is open, $y \in V$, and $f^{-1}(y)$ is
  a compact AR, then there is an open $V' \subset V$ containing $y$
  such that $(V',V)$ is a compressive pair.
\end{lemma}

\begin{proof}
  We can replace $V$ with a smaller neighborhood of $y$, so we may
  assume that $V \times V \subset W$.  Fix a contraction $c \colon
  f^{-1}(y) \times [0,1] \to f^{-1}(y)$ and a retraction $r \colon U
  \to f^{-1}(y)$ where $U \subset X$ is a neighborhood of $f^{-1}(y)$.
  Since $f^{-1}(y)$ is compact,
  $$U' = \{\, x \in U : \text{$(1 - t)x + tr(x) \in
    r_X^{-1}(f^{-1}(V))$ for all $t \in [0,1]$} \,\}$$ is an open
  neighborhood of $f^{-1}(y)$.  Lemma \ref{lemma:Nearby}, applied to a
  compact neighborhood of $f^{-1}(y)$ contained in $U'$, implies that
  there is an open $V' \subset V$ containing $y$ such that $f^{-1}(V')
  \subset U'$.  Let $\xi \colon f^{-1}(V') \times [0,1] \to f^{-1}(V)$ be
  the function
  \begin{equation*}\xi(x,t) = 
  \begin{cases}
    r_X((1 - 2t)x + 2tr(x)), & 0 \le t \le \tfrac12, \\
    c(r(x),2t - 1), & \tfrac12 \le t \le 1. 
  \end{cases} \qedhere
  \end{equation*}
\end{proof}

We conclude with an example of a map $f$ satisfying the hypotheses of
Theorem \ref{th:MainThree}, but with one fiber that is not an AR, so
that Lemma \ref{lemma:Pairs} cannot be used to verify the hypotheses
of Theorem \ref{th:MainTwo}.  Kinoshita \cite{Ki53} (see also
\cite[8.1]{McL18}) created an example that came to be known as the
\emph{tin can with a roll of toilet paper}.  This is the space $T = (D
\times \{0\}) \cup (C \times [0,1]) \cup (S \times [0,1]) \subset
\Re^3$ where: $$D = \{\, x \in \Re^2 : \|x\| \le 1\,\}, \quad C = \{\,
x \in \Re^2 : \|x\| = 1\,\}, \quad S = \{\, \tfrac{\theta}{1 +
  \theta}(\cos \theta,\sin \theta) : 0 \le \theta < \infty \,\}.$$
There is an obvious contraction of $T$ that deformation retracts
vertically onto $D \times \{0\}$, then compresses that set.  Kinoshita
gave a continuous function from $T$ to itself that does not have a
fixed point, so, in view of the Eilenberg-Montgomery fixed point
theorem, $T$ cannot be an AR.

Let $S^2 = \{\, z \in \Re^3 : \|z\| = 1 \,\}$.  It is easy to
see\footnote{Here is an explicit construction.  Let $L$ be a line in
  $\Re^2$ that passes through the origin, and let $C_L = (L \times
  \Re) \cap S^2$.  For $\varep > 0$ let $D_{L.\varep} = P_{L,\varep}
  \cup Q_{L,\varep} \cup R_{L,\varep}$ where $P_{L,\varep} = \{\, x
  \in L : \|x\| \le 1 + \varep \,\} \times \{-\varep\}$, $Q_{L,\varep}
  = \{\, x \in L : \|x\| = 1 + \varep \,\} \times [-\varep, 1 +
    \varep]$, and $R_{L,\varep}$ is the union of $\{\, x \in L : 1 \le
  \|x\| \le 1 + \varep \,\} \times \{1 + \varep\}$, the singleton
  $\{(0,0,1+\varep)\}$, and for each pair of consecutive points
  $x_0,x_1 \in L \cap S$ the set $\{\, ((1 - t)x_0 + tx_1,\max\{1 +
  \varep - \min\{t,1-t\}\|x_1 - x_0\|/\varep,\varep\}) : t \in [0,1]
  \,\}$.  It is easy to construct maps $h_{L,\varep} \colon C_L \times
  \{\varep\} \to D_{L,\varep}$ that combine to give a satisfactory
  $h$.} that there is a homeomorphism $h \colon S^2 \times (0,\infty)
\to \Re^3 \setminus T$ such that $T = \bigcap_{\varep > 0}
\overline{h(S^2 \times (0,\varep))}$.  For $x \in \Re^3 \setminus T$
let $h^{-1}(x) = (p(x),\alpha(x))$.  Let $X \subset \Re^3$ be a
compact ball centered at the origin that contains $T$, let $f \colon X
\to \Re^3$ be the function that maps each $x \in T$ to the origin and
maps each $x \notin T$ to $\alpha(x)p(x)$, and let $Y = f(X)$.
Clearly $f \colon X \to Y$ satisfies the hypotheses of Theorem
\ref{th:MainThree} but not those of Lemma \ref{lemma:Pairs}.

\bibliographystyle{aomplain}
\bibliography{all_cites}

\providecommand{\bysame}{\leavevmode\hbox to3em{\hrulefill}\thinspace}
\providecommand{\noopsort}[1]{}
\providecommand{\mr}[1]{\href{http://www.ams.org/mathscinet-getitem?mr=#1}{MR~#1}}
\providecommand{\zbl}[1]{\href{http://www.zentralblatt-math.org/zmath/en/search/?q=an:#1}{Zbl~#1}}
\providecommand{\jfm}[1]{\href{http://www.emis.de/cgi-bin/JFM-item?#1}{JFM~#1}}
\providecommand{\arxiv}[1]{\href{http://www.arxiv.org/abs/#1}{arXiv~#1}}
\providecommand{\doi}[1]{\url{https://doi.org/#1}}
\providecommand{\MR}{\relax\ifhmode\unskip\space\fi MR }
\providecommand{\MRhref}[2]{%
  \href{http://www.ams.org/mathscinet-getitem?mr=#1}{#2}
}
\providecommand{\href}[2]{#2}
\begin{thebibliography}{10}

\bibitem{AlBu03}
\bgroup\scshape{}C.~D. Aliprantis\egroup{} and
  \bgroup\scshape{}O.~Burkinshaw\egroup{}, \emph{Locally Solid Riesz Spaces
  with Applications to Economics}, second ed., American Mathematical Society,
  2003.

\bibitem{Arm69}
\bgroup\scshape{}S.~Armentrout\egroup{}, Homotopy properties of decomposition
  spaces,  \emph{Trans. Amer. Math. Soc.} \textbf{143} (1969), 499--507.

\bibitem{ArPr69}
\bgroup\scshape{}S.~Armentrout\egroup{} and \bgroup\scshape{}T.~M.
  Price\egroup{}, Decomposition into compact sets with {UV} properties,
  \emph{Trans. Amer. Math. Soc.} \textbf{141} (1969), 433--442.

\bibitem{Do80}
\bgroup\scshape{}A.~Dold\egroup{}, \emph{Lectures on Algebraic Topology},
  Springer-Verlag, New York, 1980.

\bibitem{Du51}
\bgroup\scshape{}J.~Dugundji\egroup{}, An extension of {T}ietze's theorem,
  \emph{Pacific J. Math.} \textbf{1} (1951), 353--367.

\bibitem{Dug55}
\bgroup\scshape{}J.~Dugundji\egroup{}, Remark on homotopy inverses,
  \emph{Port. Math.} \textbf{14} (1955), 39--41.

\bibitem{Du66}
\bgroup\scshape{}J.~Dugundji\egroup{}, \emph{Topology}, Allyn and Bacon, Inc.,
  Boston-London-Sydney, 1966.

\bibitem{Dyd86}
\bgroup\scshape{}J.~Dydak\egroup{}, An addendum to the {V}ietoris-{B}egle
  theorem,  \emph{Topology and Its Applications} \textbf{23} (1986), 75--86.

\bibitem{EiMo46}
\bgroup\scshape{}S.~Eilenberg\egroup{} and
  \bgroup\scshape{}D.~Montgomery\egroup{}, Fixed-point theorems for multivalued
  transformations,  \emph{Amer. J. Math.} \textbf{68} (1946), 214--222.

\bibitem{GrDu03}
\bgroup\scshape{}A.~Granas\egroup{} and \bgroup\scshape{}J.~Dugundji\egroup{},
  \emph{Fixed Point Theory}, Springer-Verlag, New York, 2003.

\bibitem{Han51}
\bgroup\scshape{}O.~Hanner\egroup{}, Some theorems on absolute neighborhood
  retracts,  \emph{Ark. Mat.} \textbf{1} (1951), 315--360.

\bibitem{Hat02}
\bgroup\scshape{}A.~Hatcher\egroup{}, \emph{Algebraic Topology}, Cambridge
  University Press, New Cambridge, 2002.

\bibitem{Ki53}
\bgroup\scshape{}S.~Kinoshita\egroup{}, On some contractible continua without
  the fixed point property,  \emph{Fund. Math.} \textbf{40} (1953), 96--98.

\bibitem{Kod55}
\bgroup\scshape{}Y.~Kodama\egroup{}, On {ANR} for metric spaces,  \emph{Sci.
  Rep. Tokyo Kyoiku Daigaku, Sect A.} \textbf{5} (1955), 96--98.

\bibitem{Law73}
\bgroup\scshape{}J.~D. Lawson\egroup{}, Comparison of taut cohomologies,
  \emph{Aequationes Math.} \textbf{9} (1973), 201--209.

\bibitem{Mar58}
\bgroup\scshape{}S.~Marde{\v s}i\'c\egroup{}, Equivalence of singular and {{\v
  C}}ech homology for {ANR}'s: Application to unicoherence,  \emph{Fund. Math.}
  \textbf{46} (1958), 29--45.

\bibitem{McL18}
\bgroup\scshape{}A.~Mc{L}ennan\egroup{}, \emph{Advanced Fixed Point Theory for
  Economics}, Springer, New York, 2018.

\bibitem{Mil59}
\bgroup\scshape{}J.~Milnor\egroup{}, On spaces having the homotopy type of a
  {CW}-complex,  \emph{Trans. Am. Math. Soc.} \textbf{90} (1959), 272--280.

\bibitem{Sma57}
\bgroup\scshape{}S.~Smale\egroup{}, A note on open maps,
  \emph{Proc.~Amer.~Math.~Soc.} \textbf{8} (1957), 391--391.

\bibitem{Sma57a}
\bgroup\scshape{}S.~Smale\egroup{}, A {V}ietoris mapping theorem for homotopy,
  \emph{Proc.~Amer.~Math.~Soc.} \textbf{8} (1957), 604--610.

\bibitem{Var58}
\bgroup\scshape{}V.~S. Varadarajan\egroup{}, Weak convergence of measures on
  separable metric spaces,  \emph{Sanky$\overline{\mbox{a}}$ A} \textbf{19}
  (1958), 15--22.

\bibitem{VoAn84}
\bgroup\scshape{}A.~J. Volovikov\egroup{} and \bgroup\scshape{}N.~L.
  Anh\egroup{}, On the {V}ietoris-{B}egle theorem,  \emph{Vestnik Moskov. Univ.
  Ser. I Math. Mekh.} \textbf{3} (1984), 70--71.

\bibitem{WJR13}
\bgroup\scshape{}F.~Waldhausen\egroup{}, \bgroup\scshape{}B.~Jahren\egroup{},
  and \bgroup\scshape{}J.~Rognes\egroup{}, \emph{Spaces of PL Manifolds and
  Categories of Simplicial Maps}, Princeton University Press, 2013.

\bibitem{Whi49}
\bgroup\scshape{}J.~Whitehead\egroup{}, Combinatorial homotopy {I},
  \emph{Bull. Am. Math. Soc.} \textbf{55} (1949), 213--245.

\end{thebibliography}

\end{document}